\documentclass[11pt]{scrartcl}

\usepackage{fourier}

\usepackage[colorlinks=true,citecolor=black,linkcolor=black,urlcolor=blue]{hyperref}

\usepackage[british]{babel}
\usepackage[utf8]{inputenc}

\usepackage{amsmath,amssymb}

\usepackage{amsthm}
\newtheorem{theorem}{Theorem}[section]
\newtheorem{corollary}[theorem]{Corollary}
\newtheorem{lemma}[theorem]{Lemma}
\newtheorem*{fact}{Fact}

\usepackage{epsfig}
\usepackage{psfrag}
\usepackage{color}

\usepackage{paralist}

\newcommand{\ground}[1]{\underline{#1}}

\let\subset\subseteq
\let\subobj\subset
\let\isom\cong

\newcommand{\INTS}{\mathbb N}
\newcommand{\Q}{\mathbb Q}

\def\one{\mathbf{1}}

\title{\bf Ramsey Properties of Permutations}

\author{Julia B\"{o}ttcher%
\thanks{%
Current address:
Department of Mathematics,
Columbia House,
London School of Economics,
Houghton Street,
London WC2A 2AE, UK.
\tt j.boettcher@lse.ac.uk}\\
\small Departamento de Ciência da Computação\\[-0.8ex]
\small Instituto de Matemática e Estatística\\[-0.8ex]
\small Universidade de São Paulo\\[-0.8ex]
\small Rua do Matão~1010\\[-0.8ex]
\small 05508--090~São Paulo, Brazil
\and
Jan Foniok%
\thanks{%
Current address:
Department of Mathematics and Statistics,
Queen's University,
Jeffery Hall,
48 University Avenue,
Kingston ON K7L 3N6,
Canada.
\tt foniok@mast.queensu.ca}\\
\small Laboratoire d'Informatique (LIX)\\[-0.8ex]
\small CNRS UMR 7161\\[-0.8ex]
\small Ecole Polytechnique\\[-0.8ex]
\small Route de Saclay\\[-0.8ex]
\small 91128~Palaiseau, France
}

\date{December 21st, 2012}

\begin{document}

\maketitle

\begin{abstract}
The age of each countable homogeneous permutation forms a Ramsey class.
Thus, there are five countably infinite Ramsey classes of permutations.

\end{abstract}

\section{Introduction}

  The finite Ramsey theorem~\cite{Ram:On-a-problem} states that for
  any natural numbers $a$, $b$ and $r$ there is a natural number
  $c$ such that in any $r$-colouring of all $a$-element subsets of
  the set $\{1,\dotsc,c\}$ there is a monochromatic $b$-element
  subset. This property is usually denoted by the \emph{Erd\H
  os--Rado partition arrow}: $c\to(b)^a_r$. In the following, the
  least such number~$c$ will be denoted by $R(a,b,r)$.

  Different variants and generalisations of this result (cf.\
  Ne\v{s}et\v{r}il~\cite{NesetrilHandbookChapter}) have been
  investigated. The concept of Ramsey classes is perhaps the most
  general of these. A \emph{Ramsey class} is a class~$\mathcal{K}$
  of objects such that for each natural number $r$ and each choice
  of objects $A,B \in \mathcal{K}$ there is an object $C \in
  \mathcal{K}$ such that in any $r$-colouring of all subobjects
  of~$C$ isomorphic to~$A$, there is a monochromatic subobject
  isomorphic to~$B$: \[C\to (B)^A_r. \]

  We consider classes of finite relational structures, with embeddings as
  subobjects (for de\-fi\-nitions see Section~\ref{sec:background}). In
  this case, there is a strong connection between Ramsey classes and
  homogeneous structures.

  \begin{theorem}[cf.\ Ne\v{s}et\v{r}il~\cite{Nes:RamseyHom}]
  \label{thm:Ramsey-hom}
    If $\mathcal{K}$ is a hereditary isomorphism-closed Ramsey class
    with the joint embedding property, then $\mathcal{K}$~is the
    age of a countable homogeneous structure.
  \end{theorem}

  For investigating Ramsey classes it therefore suffices to consider
  classes of relational structures that are the age of some homogeneous
  structure. A classification programme of Ramsey classes was
  launched by Nešetřil~\cite{Nes:RamseyHom}, see also~\cite{HubNes:Finite}.
  Ramsey classes of graphs had been characterised by
  Nešetřil~\cite{Nesetril89}, and characterisations for tournaments
  and posets appear in~\cite{Nes:RamseyHom}, and for posets again
  in~\cite{Sok:Ramsey-Properties,Sok:Ramsey-Properties-II}.

  In this paper, we make a small contribution to this classification
  programme and determine all Ramsey classes of finite permutations.
  We establish the following result.

  \begin{theorem}
  \label{thm:main}
    The age of each countable homogeneous permutation forms a Ramsey class. 
  \end{theorem}

  For the proof, we use a characterisation of homogeneous permutations by
  Cameron~\cite{Cam:Perm} (see Theorem~\ref{thm:homPermutations}). For the
  respective ages, we apply the finite Ramsey theorem, the product Ramsey
  theorem, and the partite construction and amalgamation technique of
  Nešetřil and Rödl~\cite{NesRod:Simple,NesRod:Two-proofs} in order to
  establish Theorem~\ref{thm:main} (see Sections~\ref{sec:product}
  and~\ref{sec:universal}).

\section{Background: Homogeneous structures and Ramsey classes}
\label{sec:background}
 
  According to Theorem~\ref{thm:Ramsey-hom} homogeneous relational
  structures and Ramsey classes are closely connected.  In this section we
  provide the necessary formal definitions.

  \smallskip

  We consider relational structures with a fixed signature~$\sigma$.
  The domain of a structure~$A$ is denoted by~$\ground{A}$. An
  injective mapping $f:\ground{A} \to \ground{B}$ is an \emph{embedding}
  of $A$ into $B$ if  for every $k$-ary relation symbol $R$ in~$\sigma$
  and any $k$-tuple $(x_j:j=1,\dotsc,k)$ we have $(x_j:j=1,\dots,k)
  \in R^A$ iff $(f(x_j):j=1,\dots,k) \in R^{B}$. $A$ is a
  \emph{substructure} of $B$ if the inclusion mapping is an embedding
  of $A$ into $B$; we write $A\subseteq B$. Moreover, $\binom{B}{A}$
  stands for the set of all embeddings of~$A$ into~$B$.

  Consider a class $\mathcal{K}$ of relational structures that is closed
  under isomorphisms.  For structures $A,B,C \in \mathcal{K}$ and $r \in \INTS$ the
  \emph{Erd\H{o}s-Rado partition arrow} $C \rightarrow (B)^A_r$ denotes the
  following: For each partition $\binom{C}{A}=\mathcal{A}_1\cup\dots\cup
  \mathcal{A}_r$ there exist $i\in\{1,\dots,r\}$ and $B' \subobj C$ such
  that $B' \isom B$ and $\binom{B'}{A} \subset \mathcal{A}_i$. We also call
  $B'$ a \emph{monochromatic} copy of $B$ in $C$ and the partition of
  $\binom{C}{A}$ a \emph{colouring}.  For $A \in \mathcal{K}$ we say that
  $\mathcal{K}$ is \emph{$A$-Ramsey} if for every $B\in\mathcal{K}$ and $r
  \in \INTS$ there exists a $C \in \mathcal{K}$ such that $C \rightarrow
  (B)^A_r$.  $\mathcal{K}$ is a \emph{Ramsey class} if it is $A$-Ramsey for
  every $A \in \mathcal{K}$.

  A relational structure $B$ is called \emph{homogeneous} if every
  isomorphism $f:A \to A'$ between finite substructures $A$ and
  $A'$ of $B$ can be extended to an automorphism of $B$.  For a
  structure~$B$, the \emph{age} of $B$ is
  the class of all finite structures that can be embedded into~$B$.

  Let $\mathcal{K}$ be a class of relational structures. Then
  $\mathcal{K}$ is \emph{hereditary} if $A' \in \mathcal{K}$ and
  $A \subobj A'$ imply $A \in \mathcal{K}$.
  We say that $\mathcal{K}$ has the \emph{joint embedding property}
  if for any $B,B'\in\mathcal{K}$ there is $C\in \mathcal{K}$ such
  that both $B$ and $B'$ are embeddable in~$C$.
  We say that $\mathcal{K}$ has the \emph{amalgamation property}
  if for every $A,B,B'\in\mathcal{K}$ and every embedding $f:A\to
  B$ of $A$ into $B$ and every embedding $f':A\to B'$ of $A$ into
  $B'$ there exists $C\in\mathcal{K}$ and embeddings $g:B\to C$ and
  $g':B'\to C$ such that $g \circ f = g'\circ f'$.
  If the intersection of $g(\ground{B})$ and $g'(\ground{B'})$ is
  equal to $g \circ f(\ground{A})$, we say that the amalgamation (or
  the joint embedding) is \emph{strong}.

  \smallskip

  Because of Theorem~\ref{thm:Ramsey-hom} we are interested in ages of
  countable homogeneous relational structures.  A classical result of
  Fra\"iss\'e gives a necessary and sufficient condition for a class of
  finite structures to be the age of such a structure.

  \begin{theorem}[Fra\"iss\'e~\cite{Fraisse86}]\label{thm:Fraisse}
    A class $\mathcal{K}$ of finite relational structures is the age of
    a countable homogeneous relational structure $\mathcal{U}$ iff the 
    following five conditions hold:
    \begin{compactenum}
      \item $\mathcal{K}$ is closed under isomorphism,
      \item $\mathcal{K}$ has only countably many isomorphism classes,
      \item $\mathcal{K}$ is hereditary,
      \item $\mathcal{K}$ has the joint embedding property,
      \item $\mathcal{K}$ has the amalgamation property.
    \end{compactenum}
  \end{theorem}

  For a class $\mathcal{K}$ that satisfies the conditions 1.--5.
  of Theorem~\ref{thm:Fraisse}, the countable homogeneous
  structure~$\mathcal{U}$ is unique up to isomorphism and is called
  the \emph{Fraïssé limit} of~$\mathcal{K}$.

\section{Permutations}

  From now on we will be concerned with the proof of
  Theorem~\ref{thm:main}.  We first need to introduce some
  notation. We adopt the definition of permutation following
  Cameron~\cite{Cam:Perm}: as a relational structure with two linear
  (total) orders. This definition applies to both finite and infinite
  permutations. Its advantage over viewing a permutation as a bijective
  mapping of a set to itself is that it makes the notion of subpermutation
  clearer.

  Formally, a \emph{permutation} $P=(X,{<_0},{<_1})$ consists of a set~$X$
  and two linear orders~$<_0$ and~$<_1$ of~$X$ where~$<_0$ is called the
  \emph{natural order} on~$X$ and~$<_1$ is arbitrary. An example is shown
  in Figure~\ref{fig:permutation}.  A \emph{subpermutation} of~$P$ is a
  permutation $P'=(X',{<_0'},{<_1'})$ with $X'\subseteq X$ such that~$<_i'$
  is the restriction of~$<_i$ to~$X'$.  Let $P=(X,<_0,<_1)$ be a finite
  permutation with $|X|=x$ and consider a sequence $(p_i:i=1,\dots,x)$
  of pairwise distinct numbers $p_i\in\{1,\dots,x\}$.  Then we call $(p_i:i=1,\dots,x)$
  the \emph{pattern} of $P$ if~$P$ is isomorphic to the permutation
  $P'=(\{1,\dots,x\},{<},{<_1'})$ with $i<j$ iff $p_i <_1' p_j$ for
  $i,j\in\{1,\dots,x\}$.

  \begin{figure}[ht]
     \psfrag{1}{$1$} 
     \psfrag{2}{$2$} 
     \psfrag{3}{$3$} 
     \psfrag{4}{$4$}
     \psfrag{5}{$5$}
     \psfrag{6}{$6$}
   \begin{center}
     \includegraphics[scale=.8]{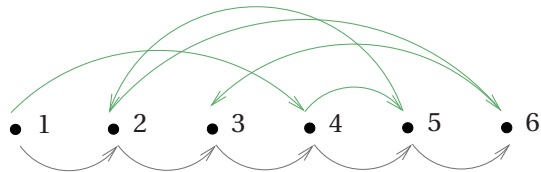}
   \end{center}
  \caption{The permutation with pattern $1,4,5,2,6,3$ (transitive arcs are omitted).}
  \label{fig:permutation}
  \end{figure}

  The countable homogeneous permutations were characterised by Cameron~\cite{Cam:Perm}.
  Let $\one$ be the unique permutation on one point. An \emph{identity permutation}
  is a permutation of the form $(X,{<_0},{<_0})$ and a \emph{reversal} is a
  permutation of the form $(X,{<_0},{>_0})$.
  An \emph{increasing sequence of decreasing sequences} is a permutation that contains
  no subpermutation with pattern $2,3,1$ and no subpermutation with pattern $3,1,2$.
  \emph{Decreasing sequences of increasing sequences} are defined analogously 
  (no $2,1,3$ and no $1,3,2$).

  \begin{theorem}[Cameron~\cite{Cam:Perm}]\label{thm:homPermutations}
    The age of each \emph{countable homogeneous permutation} is one 
    of the following classes of finite permutations:
    \begin{compactenum}
      \item $\{\one\}$,
      \item the class of identity permutations,
      \item the class of reversals,
      \item increasing sequences of decreasing sequences, 
      \item decreasing sequences of increasing sequences ,
      \item the class of all permutations.
    \end{compactenum}
  \end{theorem}

  The corresponding countable homogeneous permutations (Fraïssé
  limits), described by Cameron~\cite{Cam:Perm} as well, are the
  following:
    \begin{compactenum}
      \item $\one$,
      \item $(\Q,<,<)$,
      \item $(\Q,<,>)$,
      \item $(\Q^2,<_0,<_1)$, where $<_0$ is the lexicographic
      ordering and $<_1$~is the lexicographic ordering reversed
      within blocks,
      \item $(\Q^2,<_0,<_1)$, where $<_0$ is the lexicographic
      ordering and $<_1$~is the lexicographic ordering reversed
      between blocks,
      \item $(\Q^2,<_0,<_1)$, where $(x,y)<_0(u,v)$ if $xa+yb<ua+vb$,
      and $(x,y)<_1(u,v)$ if $xc+yd<uc+vd$;
      $a,b,c,d$ are fixed real numbers such that $b/a$ and $d/c$
      are distinct irrationals satisfying $b/a+d/c\ne0$.
    \end{compactenum}

  Our main theorem, Theorem~\ref{thm:main}, claims that each of the
  classes of Theorem~\ref{thm:homPermutations} forms a Ramsey class (and hence, because of
  Theorem~\ref{thm:Ramsey-hom}, these are the only Ramsey classes of
  permutations).  As a direct consequence of the finite Ramsey theorem we
  have the following.

  \begin{fact}
    The class of identity permutations and the class of reversals are Ramsey classes.
  \end{fact}

  In the rest of this paper we thus study the remaining three
  classes from Theorem~\ref{thm:homPermutations}.

\section{Increasing and decreasing sequences\\and the product Ramsey theorem}
\label{sec:product}

  First, we investigate the class of increasing sequences of
  decreasing sequences and the class of decreasing sequences of
  increasing sequences. For showing that they are Ramsey we make
  use of the \emph{product Ramsey theorem} of Graham, Rothschild
  and Spencer~\cite[Chapter~5.1]{GrahamRothschildSpencer90}.
  Let~$X$ be a set.
  We write $\binom{X}{k}$ for the set of all $k$-element subsets of~$X$.

  \begin{theorem}[product Ramsey theorem]

    For all $r,t,n\in\INTS$ and $p\in\INTS^t$ there exists $N\in\INTS$
    such that the following holds:
    If
    \begin{compactitem}
      \item $X_1,X_2,\dots,X_t$ are sets such that
      $|X_i| \geq N$ for all $1 \leq i \leq t$, and
      \item we $r$-colour $\binom{X_1}{p_1} \times \dots \times \binom{X_t}{p_t}$,
    \end{compactitem}
     then there are $Y_i \subset X_i$ with $|Y_i|\geq n$ for all $1 \leq i \leq t$ such that
     $\binom{Y_1}{p_1} \times \dots \times \binom{Y_t}{p_t}$
     is monochromatic.
  \end{theorem}

  In the following let $PR(r,p,t,n)$ denote the least integer $N$ such that
  the assertion of this theorem holds for `input' $r,p,t,n$.

  Next, we use the product Ramsey theorem in order to deduce the following
  variant for sets $A \subset \INTS^2$ of ordered pairs, which applies more
  directly to our problem. Two finite sets $A \subset \INTS \times J$ and
  $A' \subset \INTS \times J'$ with $J=\{j_1,\dots,j_n\}\subset \INTS$ and
  $J'=\{j'_1,\dots,j'_{n'}\} \subset \INTS$ where $j_1<\dots<j_n$ and
  $j'_1<\dots<j'_{n'}$ are \emph{isomorphic} if $n=n'$ and
  $|A\cap(\INTS\times\{j_i\})|=|A'\cap(\INTS\times\{j'_i\})|$ for all
  $1\leq i\leq n$.  Moreover, $A'$ is a \emph{substructure} of $B$ if it is a
  subset of $B$, and $\binom{B}{A}$ contains all substructures $A'$ of $B$
  that are isomorphic to~$A$.

  \begin{theorem}\label{thm:PR:modified}
    For all $r,t(A),t(B)\in\INTS$ and all $a\in\INTS^{t(A)}$, $b \in\INTS$ there exist $t(C),c\in\INTS$ such that the
    following holds:
    If
    \begin{align*}
      A &= \big\{(x,j) \colon x\in\{1,\dots,a_j\}, j\in\{1,\dots,t(A)\} \big\}\,, \\
      B &= \big\{(x,j) \colon x\in\{1,\dots,b\}, j\in\{1,\dots,t(B)\} \big\}\,, \\
      C &= \big\{(x,j) \colon x\in\{1,\dots,c\}, j\in\{1,\dots,t(C)\} \big\}\,,
    \end{align*}
    and we $r$-colour $\binom{C}{A}$, then there is a monochromatic copy of $B$ in $C$.
  \end{theorem}

  Observe that, crucially, while the sets~$B$ and~$C$ in this theorem can
  be written as the Cartesian product of two sets of integers, the first
  coordinate of pairs from~$A$ may assume different ranges depending on the
  second coordinate. We will need this property for our result
  concerning permutations.

  \begin{proof}
    Set $t(C):=R(t(A),t(B),r)$ and $\gamma:=\binom{t(C)}{t(A)}$. Moreover, define $c(1):=PR(r,a,t(A),b)$ and 
    $c(i+1):=PR(r,a,t(A),c(i))$ for $i>1$ and let $c:=c(\gamma)$.
    Let 
    \begin{alignat*}{2}
      A&=A_1\cup\dots\cup A_{t(A)} & \quad \text{ with }A_i&=\big\{(x,i) \colon
      x\in\{1,\dots,a_i\}\big\}\text{ for }1\leq i \leq t(A),\\
      B&=B_1\cup\dots\cup B_{t(B)} & \quad \text{ with }B_i&=\big\{(x,i) \colon
      x\in\{1,\dots,b\}\big\}\text{ for }1\leq i \leq t(B),\\
      C&=C_1\cup\dots\cup C_{t(C)} & \quad \text{ with }C_i&=\big\{(x,i) \colon
      x\in\{1,\dots,c\}\big\}\text{ for }1\leq i \leq t(C),
    \end{alignat*}
    and consider a colouring $\chi$ of $\binom{C}{A}$.

    We establish the theorem by finding sets $\tilde{C}_j \subset C_j$ of
    cardinality $b$ for $1 \leq j \leq t(C)$ such that
    $\tilde{C}=\tilde{C}_1\cup\dots\cup\tilde{C}_j$ has the following
    property: For all distinct $i_1,\dots,i_{t(A)} \in \{1,\dots,t(C)\}$
    all copies of~$A$ in $\tilde{C}_{i_1}\cup\dots\cup\tilde{C}_{i_{t(A)}}$
    have the same colour.  This uniquely determines an $r$-colouring of the
    set~$\mathcal{S}$ of all $t(A)$-element subsets of the integers
    $1,\dots,t(C)$. By the choice of $t(C)$ and the finite Ramsey theorem
    we find a monochromatic $t(B)$-element subset
    in~$\mathcal{S}$ 
    and accordingly a monochromatic copy of $B$ in~$C$.

    To find the sets $\tilde{C}_j$ we repeatedly apply the product Ramsey
    theorem.  The idea is to do the following step for each choice of
    distinct $i_1,\dots,i_{t(A)}\in\{1,\dots,t(C)\}$. Consider
    $C^*:=C_{i_1}\cup\dots\cup C_{i_{t(A)}}$ and choose the largest
    monochromatic subset (with respect to $\chi$) of $C^*$. For the next
    step, restrict each set $C_{i_j}$ to $C_{i_j} \cap C^*$.  Then continue
    with the next choice of $i_1,\dots,i_{t(A)}$. By the choice of $c$ and
    the product Ramsey theorem we will get subsets $\tilde{C}_j$ of ${C}_j$
    with $|\tilde{C}_j| \geq b$ at the end of this procedure.
  \end{proof} 

  This immediately implies the desired result.

  \begin{corollary}
    The class of increasing sequences of decreasing sequences and the class of decreasing sequences of increasing
    sequences are Ramsey classes.
  \end{corollary}

  \begin{proof}
    Let $A$, $B$ be two increasing sequences of decreasing sequences. Set
    $t(A)$ to be the number of decreasing sequences in~$A$ and let
    $a_j$ be the length of the $j$th decreasing sequence. Furthermore
    set $t(B)$ to be the number of decreasing sequences in~$B$ and let
    $b$ be the maximum length of a decreasing sequence in~$B$. Apply
    Theorem~\ref{thm:PR:modified} and let $C$ be the increasing sequence
    of $t(C)$ decreasing sequences of length~$c$. The proof for decreasing
    sequences of increasing sequences is analogous.
  \end{proof}

\section{The class of all permutations}
\label{sec:universal}

 To show that the class of all permutations is a Ramsey class we use the
  amalgamation technique, which was introduced by Ne\v{s}et\v{r}il and
  R\"odl~\cite{NesRod:Simple,NesRod:Two-proofs} (see also
  Ne\v{s}et\v{r}il~\cite{NesetrilHandbookChapter}).  This technique
  consists of two main parts: the \emph{partite lemma} and the
  \emph{partite construction}.  We start with some definitions.

  Let $a\in\INTS$. An \emph{$a$-partite} permutation $P=(X_1 \cup\dots\cup
  X_a, <_0, <_1)$ is a permutation on the union of disjoint sets
  $X_1,\dots,X_a$ such that $x <_0 y$ for $x \in X_i$ and $y\in X_j$
  whenever $i<j$, and $x <_0 y$ iff $x <_1 y$ for $x,y\in X_i$ (an example
  is provided in Figure~\ref{fig:partitePermutation}). The sets $X_i$ are
  called the \emph{parts} of $P$.  An $a$-partite permutation is
  \emph{transversal} if each part is of size one.  As subpermutations of
  $P$ we now consider only $a$-partite subpermutations $P'=(X'_1
  \cup\dots\cup X'_a, <'_0, <'_1)$ (where the $X'_i$ are possibly
  empty). The \emph{trace} of an $a$-partite permutation $P=(X_1
  \cup\dots\cup X_a, <_0, <_1)$ is given by $\{i : X_i\neq\emptyset\}$.

  \begin{figure}[ht]
     \psfrag{1}{$1$} 
     \psfrag{2}{$2$} 
     \psfrag{3}{$3$} 
     \psfrag{4}{$4$}
     \psfrag{5}{$5$}
     \psfrag{6}{$6$}
     \psfrag{7}{$7$} 
     \psfrag{8}{$8$} 
     \psfrag{9}{$9$} 
     \psfrag{10}{$10$}
     \psfrag{11}{$11$}
     \psfrag{12}{$12$}
     \psfrag{13}{$13$}
   \begin{center}
     \includegraphics[scale=.8]{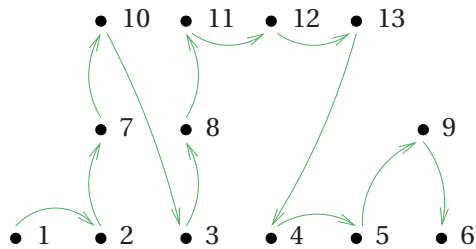}
   \end{center}
  \caption{The $3$-partite permutation with pattern $1,2,7,10,3,8,11,12,13,4,5,9,6$.}
  \label{fig:partitePermutation}
  \end{figure}


\subsection{The Partite Lemma}

 Following the strategy of Ne\v{s}et\v{r}il and R\"odl we first formulate a
  Ramsey lemma concerning copies of $a$-partite transversals in arbitrary
  $a$-partite permutations.  In our case this lemma asserts the following.

  \begin{lemma}[partite lemma]
  \label{lem:part}
    Let $A$ be an $a$-partite transversal and $B$ an arbitrary $a$-partite permutation. Then for any $r\in\INTS$
    there exists an $a$-partite permutation $C$ such that $C \rightarrow (B)^A_r$.
  \end{lemma}

  For the proof of this lemma we use the following class of special
  permutations.  An $a$-\emph{snake} $P=(X_1 \cup\dots\cup X_a, <_0, <_1)$
  of \emph{length} $\xi=\xi(P)$ is an $a$-partite permutation with parts
  \begin{equation*}
    X_i=\big\{ (i-1)\xi+x \colon 1 \leq x \leq \xi \big\} =: \{1_i,\dots,\xi_i\}
  \end{equation*}
  such that $x_i <_1 x_j $ for $0\leq i<j\leq a$ and $1 \leq x \leq \xi$,
  and $x_a <_1 y_1$ for $y=x+1$ and $1 \leq x < \xi$ (see
  Figure~\ref{fig:snake}).

  \begin{figure}[ht]
     \psfrag{1}{$1$} 
     \psfrag{2}{$2$} 
     \psfrag{3}{$3$} 
     \psfrag{4}{$4$}
     \psfrag{5}{$5$}
     \psfrag{6}{$6$}
     \psfrag{7}{$7$} 
     \psfrag{8}{$8$} 
     \psfrag{9}{$9$} 
     \psfrag{10}{$10$}
     \psfrag{11}{$11$}
     \psfrag{12}{$12$}
   \begin{center}
     \includegraphics[scale=.8]{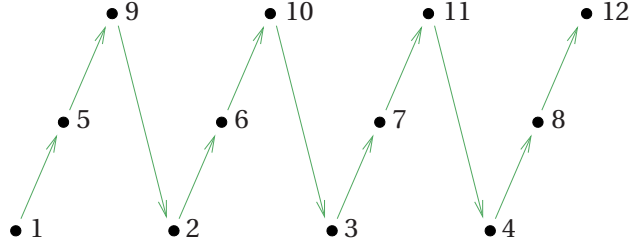}
   \end{center}
  \caption{The $3$-partite snake of length $4$.}
  \label{fig:snake}
  \end{figure}

  \begin{fact}
    Each $a$-partite permutation $P=(X_1 \cup\dots\cup X_a, <_0, <_1)$ is a subpermutation of
    the $a$-snake of length $\sum_i |X_i|$.
  \end{fact}

  For an arbitrary $a$-partite permutation $P$ let $\xi(P)$ be the length of the shortest $a$-snake
  that contains $P$ as a subpermutation.
  By the previous fact the partite lemma, Lemma~\ref{lem:part}, is a direct consequence of the following lemma.

  \begin{lemma}
    Let $A$ be an $a$-partite transversal and $B$ an arbitrary $a$-snake. Then for any $r\in\INTS$
    there exists an $a$-snake $C$ such that $C \rightarrow (B)^A_r$.   
  \end{lemma}
  \begin{proof}
    We proceed by induction on $a$. The base case $a=1$ is trivial. 
    For the induction step let $\xi(r,A',B')$ be the length of a
    snake $C$ such that the lemma holds for $r$~colours and
    $(a-1)$-partite $A'$ and $B'$. Now, consider an $a$-partite
    transversal $A$ and an $a$-snake $B$ and assume that the first element
    of $A$ is in
    part $\ground{A}_1$. The $(a-1)$-partite permutation induced on $\ground{A}\setminus\ground{A}_1$
    is denoted by $A^*$ (and the trivial $1$-partite permutation induced
    on~$\ground{A}_1$ by~$A_1$).
    Let $C$ be the $a$-partite snake of length
    \begin{equation*}
      \underset{(r-1)\xi(B)+1\,\mbox{times}}{
        \underbrace{
          \xi(r,A^*,\xi(r,A^*,\xi( \dots (r,A^*,\xi(
        }
      }
      r,A^*, A^*)+1) \dots )+1)+1)
    . 
    \end{equation*}
    We claim that $C$ has the desired properties. Indeed, consider a
    colouring $\chi$ of $\binom{C}{A}$.  Let $\ground{C}_1$ be the part of
    $C$ that corresponds to $\ground{A}_1$, let~$C_1$ be the $1$-partite
    permutation induced on~$\ground{C}_1$, and $C^*$ the $(a-1)$-partite
    permutation induced on $\ground{C}\setminus\ground{C}_1$.  Choose the
    first element $c_1$ of~$C_1$ and consider copies of $A$ in $C$ that
    start in $c_1$.  These copies of $A$ are in one to one correspondence
    with copies of $A^*$ in $C^*$. Let $\chi^*$ be the corresponding
    colouring of $\binom{C^*}{A^*}$, i.e., $\chi^*(A^*):=\chi({A^* \cup
    c_1})$. By the choice of $C$ and the induction hypothesis we find an
    $(a-1)$-snake in~$C^*$ that is monochromatic under~$\chi^*$ and has
    length
    \begin{equation*}
      \underset{(r-1)\xi(B)\,\mbox{times}}{
        \underbrace{
          \xi(r,A^*,\xi(r,A^*,\xi( \dots (r,A^*,\xi(
        }
      }
      r,A^*, A^* )+1) \dots )+1)+1) \,.
    \end{equation*}
    We continue by restricting $C$ to the elements of this snake and
    appropriate intermediate elements in $C_1$ and repeat this process
    $(r-1)\xi(B)+1$ times. Let $\tilde{C}$ be the $a$-partite snake that
    has in part $\ground{C}_1$ all the elements that were chosen as
    $c_1$ in this process. By construction,
    $\xi(\tilde{C})=(r-1)\xi(B)+1$. Observe, moreover, that the colour of a
    copy of $A$ in $\tilde{C}$ depends only on its element in $A_1$. By the
    pigeon hole principle we therefore get a monochromatic copy of $B$ in
    $\tilde{C}$.
  \end{proof}

\subsection{The Partite Construction}

  For the partite construction we additionally need the following lemma.

  \begin{lemma}\label{lem:perm:partiteAmalgam}
    For each $a\in\INTS$ the class of $a$-partite permutations has the amalgamation property. 
  \end{lemma}

  \begin{proof}
  It is easy to see that the class of all permutations
  has strong amalgamation (recall the definition from
  Section~\ref{sec:background}). Hence, to amalgamate $a$-partite
  permutations, amalgamate
  the underlying permutations strongly in such a way that within one part
  the two orderings coincide. Membership in parts is preserved.
  \end{proof}

  With this we are now ready to deduce from the partite lemma via the
  partite construction that the class of all permutations is Ramsey.

  \begin{theorem}
    The class of all permutations is a Ramsey class.
  \end{theorem}

  \begin{proof}
    Let $r\ge 2$ be an integer and $A$ and $B$ be arbitrary permutations of
    sizes $a$ and $b$, respectively. We need to show that there is a
    permutation~$C$ with $C\rightarrow(B)^A_r$.  For this purpose let
    $q=R(a,b,r)$ and $\gamma=\binom{q}{a}$.  In the partite construction we
    will recursively define permutations $P_0$, \dots, $P_{\gamma}$ and
    $P_\gamma$ will be the permutation~$C$ in quest.

    We consider $A$ as an $a$-partite permutation and $B$ as a $b$-partite
    permutation. Let $P_0$ be a $q$-partite permutation such that any
    arbitrarily chosen $b$ parts of $P_0$ induce a copy of~$B$. Such a
    permutation can easily be constructed, for instance, by concatenating
    $\binom{q}{b}$ copies of~$B$ `spread' over the right parts.

    Let $\binom{\{1,\dotsc,q\}}{a}=\{M_1,\dotsc,M_{\gamma}\}$.  For
    $i\geq1$, each permutation $P_{i}$ is then obtained from $P_{i-1}$ by
    performing the following construction: Let $D_i$ be the permutation
    induced on the parts enumerated by~$M_i$. According to the partite lemma,
    Lemma~\ref{lem:part}, there is a $b$-partite permutation $E_i$ such
    that $E_i \rightarrow (D_i)^A_r$.  Construct $P_{i}$ by taking
    $\left|\binom{E_i}{D_i}\right|$ copies of $P_{i-1}$ and amalgamating
    them along the copies of $D_i$ in $E_i$. This is possible by
    Lemma~\ref{lem:perm:partiteAmalgam}.

  Then $P_{\gamma}\rightarrow(B)^A_r$. This is shown by backward
  induction: Consider a colouring $\chi$ of $P_\gamma$ and let
  $P^*_\gamma=P_\gamma$.  From $P^*_{i+1}$ choose a copy $P^*_i$ of $P_i$
  such that all copies of $A$ in the corresponding copy of $D_i$ have
  the same colour. $P^*_i$~exists by construction (because $E_i \rightarrow
  (D_i)^A_r$). 
  Eventually, we obtain a copy $P^*_0$ of~$P_0$ with the following
  property: The colour of a copy of~$A$ in~$P^*_0$ depends only on its
  trace. Accordingly the colouring $\chi$ of $P^*_0$ induces a colouring of
  subsets of size $a$ of the first $q$ positive integers.  By the choice
  of~$q$ and the Ramsey theorem we therefore find a monochromatic copy
  of~$B$ in this copy of~$P_0$.  \end{proof}

\section{Concluding remarks}

We have recently been informed that some of our results have independently
been obtained by Sokić~\cite{Sok:Ramsey,Sok:Ramsey-property}.
An alternative proof of our results, using topological properties
of automorphism groups, can be obtained as an application
of Bodirsky's recent manuscript~\cite{Bod:New-Ramsey-Classes}.

The class of all structures consisting of 3 or more total orders
is also a Ramsey class; this can be proved either by a modification
of our proof or by the cross-construction of Sokić~\cite{Sok:Ramsey}.
However, at present the classification of all homogeneous structures
with $k$ total orders is open for $k\ge3$.

The partite construction seems to be a promising method for proving
Ramsey properties of more general amalgamation classes. The
applications we have in mind include, e.g., classes of structures
defined by forbidding the existence of a homomorphism from a set
of finite connected structures. Such classes were studied by Cherlin,
Shelah and Shi~\cite{CheSheShi:Universal}. For finite sets of
forbidden structures, the solution was announced by
Nešetřil~\cite{Nes:CSS-Ramsey}. 

\section*{Acknowledgements}

This research was in part done during the 2005 Prague Doccourse `Modern
Methods in Ramsey Theory', supported by the EU Research Network COMBSTRU
and by DIMATIA. We would like to thank the organisers for an instructive
and inspiring event, and the funding bodies for financial support. In
particular we would like to thank Jarik Nešetřil for suggesting the problem
to work on and for intense discussions.

The first author was financially supported by CNPq (Proc.~484154/2010-9) and
by 
FAPESP (Proc.~2009/17831-7), and is grateful to NUMEC/USP, Núcleo de
Modelagem Estocástica e Complexidade of the University of São Paulo, for
supporting this research.

The second author was supported by ERC Starting Grant `CSP-Complexity'
no.~257039 and by the Institute for Theoretical Computer Science
(ITI), project 1M0545 of the Ministry of Education of the Czech
Republic.

\end{document}